\documentclass[11pt]{article}

\usepackage[margin=1in]{geometry}
\usepackage{amsmath,amssymb,amsthm,mathtools}
\usepackage{microtype}
\usepackage{setspace}
\usepackage[colorlinks=true,linkcolor=blue,citecolor=blue,urlcolor=blue]{hyperref}

\allowdisplaybreaks

\theoremstyle{plain}
\newtheorem{theorem}{Theorem}[section]
\newtheorem{lemma}[theorem]{Lemma}
\newtheorem{proposition}[theorem]{Proposition}

\theoremstyle{remark}
\newtheorem{remark}[theorem]{Remark}

\DeclareMathOperator{\sgn}{sgn}
\DeclareMathOperator{\rk}{rk}

\title{Real-rootedness of Kazhdan--Lusztig and $Z$-polynomials of thagomizer matroids and graphic matroids of $K_{2,n}$}
\author{Philip B. Zhang\\[2pt]
\small College of Mathematical Sciences, Tianjin Normal University\\
\small Tianjin 300387, P. R. China\\
\small \texttt{zhang@tjnu.edu.cn}}
\date{}
\hypersetup{
  pdfauthor={Philip B. Zhang},
  pdftitle={Real-rootedness of Kazhdan--Lusztig and Z-polynomials of thagomizer matroids and graphic matroids of K2,n},
  pdfkeywords={Kazhdan--Lusztig polynomial, Z-polynomial, real-rootedness, thagomizer matroid}
}

\begin{document}
\setstretch{1.08}
\maketitle

\begin{abstract}
Let $T_n=K_{1,1,n}$, and let $P_n(x)$ denote the Kazhdan--Lusztig polynomial of its graphic matroid. We prove that, whenever $n\ge2$ and $0\le\lambda\le n/2$, the polynomial $P_n(x)+\lambda x$ has exactly $\lfloor n/2\rfloor$ zeros, all of which are negative and simple.  In particular, the Kazhdan--Lusztig polynomials of the graphic matroids of $T_n$ and $K_{2,n}$ are real-rooted. We also prove that, for $n\ge2$, the common polynomial $Z_{T_n}(x)=Z_{K_{2,n}}(x)$ has $n+1$ distinct negative zeros. The proofs use a common rational transformation, reducing the Kazhdan--Lusztig case to alternating sign evaluations at the zeros of a Chebyshev polynomial and the $Z$-polynomial case to a unit-circle criterion for self-inversive polynomials.
\end{abstract}


\section{Introduction}

The objective of this paper is to study the real-rootedness of the
Kazhdan--Lusztig and $Z$-polynomials of the thagomizer matroids and of
the graphic matroids of $K_{2,n}$.  Our main result is in fact a
one-parameter statement: the two Kazhdan--Lusztig polynomials occur in
a polynomial pencil that is real-rooted throughout the full parameter
interval $0\le\lambda\le n/2$, with all zeros simple and negative.

The Kazhdan--Lusztig polynomial $P_M(x)$ of a matroid $M$ was
introduced by Elias, Proudfoot, and Wakefield~\cite{EliasProudfootWakefield}.
Its coefficients are nonnegative by the singular Hodge theory of
Braden, Huh, Matherne, Proudfoot, and Wang~\cite{BradenHuhMatherneProudfootWang};
see also the ICM survey of Braden and Proudfoot~\cite{BradenProudfootICM}.
Real-rootedness is a considerably stronger property, since it implies
log-concavity by Newton's inequalities.  On the other hand, it cannot
be expected for arbitrary matroids: Cheng and Liu~\cite{ChengLiuNonunimodal}
constructed representable matroids over every finite field whose
Kazhdan--Lusztig polynomials are not even unimodal.  This raises the
natural question of which matroids have real-rooted Kazhdan--Lusztig
polynomials.

Real-rootedness has been established for several classes.  If
$U_{m,d}$ denotes the uniform matroid of rank $d$ on $m+d$ elements,
Gao, Lu, Xie, Yang, and Zhang~\cite{GaoLuXieYangZhangUniform} proved
that its Kazhdan--Lusztig and $Z$-polynomials are real-rooted for
$2\le m\le15$ and $d\ge1$.  For thagomizer matroids, Wu and
Zhang~\cite{WuZhangThagomizerLogConcavity} established the
log-concavity of the coefficients.  We prove that their
Kazhdan--Lusztig polynomials are real-rooted and, more precisely, that
all their zeros are simple and negative.

For $n\ge0$, let $T_n=K_{1,1,n}$.  The graphic matroid of $T_n$ is
called the $n$th thagomizer matroid.  The graph $T_n$ is obtained from
$K_{2,n}$ by adding the edge joining the two vertices in the part of
size $2$.  We use the same notation for a graph and its graphic
matroid, and write
\[
P_n(x)=P_{T_n}(x),\qquad P_0(x)=1.
\]
Let $C(z)=\sum_{n\ge0}C_nz^n=(1-\sqrt{1-4z})/(2z)$ be the Catalan
generating function.  Gedeon~\cite[Theorem~1(2)]{Gedeon} obtained
\begin{equation}\label{eq:ThagGF}
   \sum_{n\ge0}P_n(x)z^n=C\bigl(z-(1-x)z^2\bigr).
\end{equation}
Moreover, Gedeon, Proudfoot, and Young~\cite[Theorem~5.8]{GPY}
proved that $P_{K_{2,n}}(x)=P_n(x)+x$ for $n\ge2$.  These two
formulas lead us to consider the pencil $P_n(x)+\lambda x$.

Our first theorem gives a uniform range for the parameter $\lambda$.

\begin{theorem}\label{thm:KLpencil}
For $n\ge2$ and $0\le\lambda\le n/2$, the polynomial
$P_n(x)+\lambda x$ has exactly $\lfloor n/2\rfloor$ zeros, all of
which are negative and simple.
\end{theorem}

The interior nodes used in its proof are the zeros of the Chebyshev
polynomial $U_{n-1}$ of the second kind.  After a degree-$n$
normalization of the pencil, alternating evaluations at these nodes
give strict interlacing with $U_{n-1}$ and hence the exact zero count.

Taking $\lambda=0$ and $\lambda=1$, we obtain the two graphic
families that motivated the pencil.

\begin{theorem}\label{thm:KLtwoFamilies}
For every $n\ge2$, each of the polynomials $P_n(x)$ and
$P_{K_{2,n}}(x)$ has exactly $\lfloor n/2\rfloor$ zeros, all of
which are negative and simple.
\end{theorem}

We next turn to the $Z$-polynomial introduced by Proudfoot, Xu, and
Young~\cite{ProudfootXuYoung}.  Ferroni, Nasr, and
Vecchi~\cite{FerroniNasrVecchiGamma} derived a formula for
$Z_{T_n}(x)$ in the proof of their Proposition~5.18 and proved in
Proposition~5.20 that $Z_{T_n}(x)=Z_{K_{2,n}}(x)$ for $n\ge2$.
Their Proposition~5.3 relates the real-rootedness of a palindromic
polynomial with positive coefficients to the negative real-rootedness
of its $\gamma$-polynomial.  In the present case, we prove that the
corresponding $\gamma$-polynomial has only simple negative zeros and
then transfer this simplicity to the $Z$-polynomial.

\begin{theorem}\label{thm:ZRealRooted}
For every $n\ge2$, the polynomial
$Z_{T_n}(x)=Z_{K_{2,n}}(x)$ has $n+1$ distinct negative zeros.
\end{theorem}

The polynomial $P_n(x)$ also admits a permutation interpretation.
Vella's descent enumeration for $321$-avoiding
permutations~\cite{Vella}, together with~\eqref{eq:ThagGF}, gives
\[
P_n(x)=\sum_{\pi\in\mathfrak S_n(321)}
       x^{\operatorname{des}(\pi)}.
\]
Here $\mathfrak S_n(321)$ denotes the set of $321$-avoiding
permutations in $\mathfrak S_n$, and $\operatorname{des}(\pi)$ is the
number of descents of $\pi$.  When $n=0$, the sum contains the empty
permutation, whose descent number is $0$.

The proofs of Theorems~\ref{thm:KLpencil} and~
\ref{thm:ZRealRooted} are based on the same rational substitution
\[
w=\left(\frac{q-1}{q+1}\right)^2.
\]
For the pencil $P_n(x)+\lambda x$, we first transform $P_n(x)$
into a palindromic polynomial whose interior coefficients are
nonpositive.  The contribution of $\lambda x$ can then be evaluated
explicitly at roots of unity, yielding alternating signs on the
negative real axis.  After normalization, these signs lift to the
Chebyshev nodes, where strict interlacing supplies the zero-counting
step in the proof of Theorem~\ref{thm:KLpencil}.  The generating
function also reveals a direct Catalan--Chebyshev relation.
For the $Z$-polynomial, the same substitution produces a
self-inversive polynomial, and a theorem of Lakatos and Losonczi can
be applied.  Section~\ref{sec:KLproof} is devoted to the polynomial
pencil, while Section~\ref{sec:Zpolynomials} treats the
$Z$-polynomials.

\section{The Kazhdan--Lusztig polynomial pencil}\label{sec:KLproof}

We prove Theorem~\ref{thm:KLpencil} by transforming $P_n$ into a
palindromic polynomial and evaluating it at roots of unity.  The
transformed polynomials have a simple generating function, and their
coefficient signs reduce to one mixed-coefficient estimate for the
Catalan series.  The resulting signs are then lifted to the Chebyshev
nodes, where strict interlacing gives the exact zero count; finally,
the zeros are transferred back to the negative real axis.

\subsection{The transformed polynomials}

For an integer $N\ge0$, a polynomial
$F(q)=\sum_{r=0}^{N}f_rq^r$ is said to be \emph{palindromic with
respect to $N$} if $f_r=f_{N-r}$ for all $r$, or equivalently,
$F(q)=q^NF(q^{-1})$.  If $f_0\ne0$, then $N=\deg F$, and we simply
say that $F$ is palindromic.

\begin{lemma}\label{lem:PhiBasic}
Let $N\ge1$, and let
$A(x)=\sum_{m=0}^{\lfloor N/2\rfloor}\alpha_mx^m$.  Define
\begin{equation}\label{eq:PhiDef}
\Phi_N(A)(q)=\sum_{m=0}^{\lfloor N/2\rfloor}
\alpha_m(q-1)^{2m}(1+q)^{N-2m}.
\end{equation}
Then $\Phi_N(A)$ is palindromic with respect to $N$.  Moreover, for
$q=e^{2i\theta}$ with $0\le\theta<\pi/2$,
\begin{equation}\label{eq:PhiUnitCircle}
\Phi_N(A)(e^{2i\theta})
=2^Ne^{iN\theta}\cos^N\theta\,A(-\tan^2\theta).
\end{equation}
\end{lemma}

\begin{proof}
Each summand in \eqref{eq:PhiDef} is transformed into $q^{-N}$ times
itself under $q\mapsto q^{-1}$, so $\Phi_N(A)$ is palindromic with
respect to $N$.  For $q\ne-1$,
\begin{equation}\label{eq:PhiRational}
\Phi_N(A)(q)=(1+q)^N
A\!\left(\left(\frac{q-1}{q+1}\right)^2\right).
\end{equation}
If $q=e^{2i\theta}$, then $(q-1)/(q+1)=i\tan\theta$ and
$1+q=2e^{i\theta}\cos\theta$, which gives
\eqref{eq:PhiUnitCircle}.
\end{proof}

We now apply this transformation to the thagomizer polynomials.  Since
$T_k$ has rank $k+1$, the defining degree condition for matroid
Kazhdan--Lusztig polynomials gives
$\deg P_k<(k+1)/2$, and hence $\deg P_k\le\lfloor k/2\rfloor$.
Thus $\Phi_k(P_k)$ is well defined for $k\ge1$.  Set
$B_k(q)=\Phi_k(P_k)(q)$ for $k\ge1$ and $B_0(q)=1$.  For $q\ne-1$,
put $v=((q-1)/(q+1))^2$.  By \eqref{eq:PhiRational} and
\eqref{eq:ThagGF},
\begin{align*}
\sum_{k\ge0}B_k(q)z^k
&=\sum_{k\ge0}P_k(v)((1+q)z)^k\\
&=C\!\left((1+q)z-(1-v)(1+q)^2z^2\right)\\
&=C((1+q)z-4qz^2).
\end{align*}
For each $k$, the coefficients of $z^k$ on the two sides are
polynomials in $q$ and agree for every $q\ne-1$.  They therefore agree
identically in $q$, and hence
\begin{equation}\label{eq:Bgenfun}
\sum_{k\ge0}B_k(q)z^k=C((1+q)z-4qz^2).
\end{equation}

Write $B_k(q)=\sum_{r=0}^k\beta_{k,r}q^r$, and put
$\mathcal C(X,Y)=C(X+Y-4XY)$.  Substituting $X=z$ and $Y=qz$ in
\eqref{eq:Bgenfun} gives
\begin{equation}\label{eq:betaMixedCoeff}
\beta_{k,r}=[X^{k-r}Y^r]\mathcal C(X,Y).
\end{equation}
Thus the coefficient signs of $B_k$ are reduced to the mixed
coefficients of $\mathcal C$.  The next lemma supplies exactly the
mixed-coefficient estimate needed later at roots of unity.

\subsection{A mixed-coefficient estimate}

By a mixed coefficient of a bivariate series we mean a coefficient
$[X^aY^b]$ with $a,b\ge1$.  Let $U=C(X)-1$ and $V=C(Y)-1$.

\begin{lemma}\label{lem:twoVariableCatalan}
One has
\begin{equation}\label{eq:CXYdecomp}
C(X+Y-4XY)
=C(X)+C(Y)-1-\frac{UV(U+V)}{1+UV}.
\end{equation}
Consequently, $[X^aY^b]C(X+Y-4XY)\le0$ whenever $a,b\ge1$.
\end{lemma}

\begin{proof}
Since $U=X(1+U)^2$ and $V=Y(1+V)^2$, we have
$\sqrt{1-4X}=(1-U)/(1+U)$ and
$\sqrt{1-4Y}=(1-V)/(1+V)$.  All square roots here denote the unique
formal power series with constant term $1$.  Together with
$1-4(X+Y-4XY)=(1-4X)(1-4Y)$, this gives
\[
C(X+Y-4XY)=\frac{(1+U)(1+V)}{1+UV}
=1+\frac{U+V}{1+UV},
\]
which is equivalent to \eqref{eq:CXYdecomp}.

It remains to show that $UV(U+V)/(1+UV)$ has nonnegative mixed
coefficients.  As usual, we set $\binom{m}{r}=0$ when $r<0$ or $r>m$.
In particular, all sums below are finite.  Since $UV$ has zero
constant term, its
formal geometric expansion is
\[
\frac{UV}{1+UV}=\sum_{\ell\ge1}(-1)^{\ell-1}U^\ell V^\ell.
\]
Since $U=X(1+U)^2$ and $V=Y(1+V)^2$, Lagrange inversion gives,
for $a,\ell\ge1$, the identities
$[X^a]U^\ell=(\ell/a)\binom{2a}{a-\ell}$ and
$[Y^b]V^\ell=(\ell/b)\binom{2b}{b-\ell}$.  Hence, for $a,b\ge1$,
\[
[X^aY^b]\frac{UV}{1+UV}
=\sum_{\ell\ge1}(-1)^{\ell-1}\frac{\ell^2}{ab}
\binom{2a}{a-\ell}\binom{2b}{b-\ell}.
\]
Let
$W_k=(-1)^k\binom{2a}{a+k}\binom{2b}{b-k}$.  The von Szily identity
for the super-Catalan numbers, in the form recorded by
Gessel~\cite[Section~6]{GesselSuperBallot}, states that
\[
\sum_{k\in\mathbb Z}W_k
=\frac{(2a)!(2b)!}{a!b!(a+b)!}.
\]
Set $R_k=-\tfrac12(k+a)(k+b)W_k$ and
$m=\min\{a,b\}$.  Then $W_k$ is supported on $-m\le k\le m$.
For every index in this range, $W_k\ne0$, and
$W_{k+1}/W_k=-(a-k)(b-k)/((a+k+1)(b+k+1))$.
At $k=m$, the quotient identity remains valid after setting
$W_{m+1}=0$, since one of the factors $a-k$ and $b-k$ vanishes.
Consequently,
\[
R_{k+1}-R_k
=\frac12\bigl((a-k)(b-k)+(a+k)(b+k)\bigr)W_k
=(k^2+ab)W_k
\qquad(-m\le k\le m).
\]
Moreover, one factor in $R_{-m}$ vanishes and $W_{m+1}=0$, so
$R_{-m}=R_{m+1}=0$.  The recurrence therefore telescopes over the
full support and gives
$\sum_{k\in\mathbb Z}k^2W_k=-ab\sum_{k\in\mathbb Z}W_k$.
By the symmetry of the binomial
coefficients, $W_{-\ell}=W_\ell$, and therefore
\[
\begin{aligned}
[X^aY^b]\frac{UV}{1+UV}
&=-\frac1{ab}\sum_{\ell\ge1}\ell^2W_\ell
=-\frac1{2ab}\sum_{k\in\mathbb Z}k^2W_k\\
&=\frac12\sum_{k\in\mathbb Z}W_k
=\frac12\frac{(2a)!(2b)!}{a!b!(a+b)!}>0.
\end{aligned}
\]
Thus $UV/(1+UV)$ has positive mixed coefficients and no terms supported
on either coordinate axis.  Since $U+V$ has nonnegative coefficients,
the product $UV(U+V)/(1+UV)$ has nonnegative mixed coefficients.  The
conclusion follows from \eqref{eq:CXYdecomp}.
\end{proof}

The preceding lemma now determines the signs of all interior
coefficients of $B_n$.  Combined with palindromicity, this turns the
unit-circle evaluations of $B_n$ into alternating signs for the
original polynomial on the negative real axis.

\subsection{Proofs of the main results}

\begin{proof}[Proof of Theorem~\ref{thm:KLpencil}]
Fix $n\ge2$ and $0\le\lambda\le n/2$.  Put
$P_{n,\lambda}(x)=P_n(x)+\lambda x$ and
$d=\lfloor n/2\rfloor$.

\medskip
\noindent\emph{Step 1: Coefficient signs and a root-of-unity bound.}
By \eqref{eq:betaMixedCoeff}, the symmetry of $\mathcal C(X,Y)$ gives
$\beta_{n,r}=\beta_{n,n-r}$.  Moreover,
$\mathcal C(X,0)=C(X)$ and $\mathcal C(0,Y)=C(Y)$, while
Lemma~\ref{lem:twoVariableCatalan} gives
$\beta_{n,0}=\beta_{n,n}=C_n$ and
$\beta_{n,r}\le0$ for $1\le r\le n-1$.
Setting $q=1$ in \eqref{eq:Bgenfun} yields
$\sum_{k\ge0}B_k(1)z^k=C(2z-4z^2)=(1-2z)^{-1}$.  Indeed,
$1-4(2z-4z^2)=(1-4z)^2$, whose formal square root with constant
term $1$ is $1-4z$.  Hence
$B_n(1)=2^n$.

For $j\in\mathbb Z$, let $q=e^{2\pi ij/n}$.  Since $B_n$ has real
coefficients and is palindromic with respect to $n$, we have
$B_n(q)=q^nB_n(q^{-1})=B_n(\overline q)=\overline{B_n(q)}$; hence
$B_n(q)$ is real.  Taking real parts in
$B_n(q)-B_n(1)=\sum_{r=1}^{n-1}\beta_{n,r}(q^r-1)$ gives
\begin{align}
B_n(q)-B_n(1)
&=\sum_{r=1}^{n-1}\beta_{n,r}
\left(\cos\frac{2\pi jr}{n}-1\right) \notag\\
&=2\sum_{r=1}^{n-1}(-\beta_{n,r})
\sin^2\frac{\pi jr}{n}\ge0.
\label{eq:BnRootUnityBound}
\end{align}
Thus $B_n(q)\ge2^n$ for every $n$th root of unity $q$.  This
uniform lower bound is the input needed to control the additional
term arising from $\lambda x$.

\medskip
\noindent\emph{Step 2: Alternating signs on the negative real axis.}
Let $0\le j\le d$ when $n$ is odd, and let $0\le j\le d-1$ when
$n$ is even.  Put
$\theta=j\pi/n$ and $q=e^{2i\theta}$.  By linearity of $\Phi_n$ and
\eqref{eq:PhiUnitCircle},
\begin{equation}\label{eq:pencilTransformed}
2^n(-1)^j\cos^n\theta\,
P_{n,\lambda}(-\tan^2\theta)
=B_n(q)-2^n\lambda(-1)^j
\sin^2\theta\cos^{n-2}\theta.
\end{equation}
If $j$ is odd, the second term on the right-hand side of
\eqref{eq:pencilTransformed} is nonnegative, and the whole expression
is positive by \eqref{eq:BnRootUnityBound}.  If $j$ is even, the
right-hand side of \eqref{eq:pencilTransformed} is at least
$2^n(1-\lambda\sin^2\theta\cos^{n-2}\theta)$.  We claim that
\begin{equation}\label{eq:perturbBound}
0\le\sin^2\theta\cos^{n-2}\theta<\frac2n.
\end{equation}
If $n=2$, then $d=1$ and the only relevant index is $j=0$, so the
left-hand side is $0$.  For $n>2$, put
$u=\cos^2\theta$.  The maximum of
$(1-u)u^{(n-2)/2}$ on $0\le u\le1$ is attained at
$u=(n-2)/n$ and equals
$\frac2n((n-2)/n)^{(n-2)/2}<2/n$.  By
\eqref{eq:perturbBound} and $0\le\lambda\le n/2$, the right-hand
side of \eqref{eq:pencilTransformed} is positive.  Therefore
\begin{equation}\label{eq:pencilSign}
\sgn P_{n,\lambda}\!\left(-\tan^2\frac{j\pi}{n}\right)=(-1)^j.
\end{equation}
These evaluation points lie on the nonpositive real axis and are
strictly decreasing as $j$ increases.  We now lift these signs to the
standard cosine nodes associated with a Chebyshev polynomial; this
supplies the zero-counting step.

\medskip
\noindent\emph{Step 3: Degree, Chebyshev interlacing, and zero counting.}
Expanding $C(t)=\sum_{r\ge0}C_rt^r$ in \eqref{eq:ThagGF} and
extracting the coefficient of $z^n$ gives
\begin{equation}\label{eq:PnExpansion}
P_n(x)=\sum_{m=0}^{\lfloor n/2\rfloor}
C_{n-m}\binom{n-m}{m}(x-1)^m.
\end{equation}
By \eqref{eq:PnExpansion}, if $d\ge2$, then
$[x^d]P_{n,\lambda}(x)=C_{n-d}\binom{n-d}{d}>0$.  If $d=1$, then
$[x]P_{n,\lambda}(x)=(n-1)C_{n-1}+\lambda>0$.  Thus
$\deg P_{n,\lambda}=d$ and its leading coefficient is positive.

Write $P_{n,\lambda}(x)=\sum_{m=0}^d a_mx^m$ and define its
degree-$n$ normalization by
\begin{equation}\label{eq:normalizedP}
\widetilde P_{n,\lambda}(y)
=\sum_{m=0}^d a_m y^{n-2m}(y^2-1)^m.
\end{equation}
For $y\ne0$,
\begin{equation}\label{eq:normalizedPRational}
\widetilde P_{n,\lambda}(y)
=y^nP_{n,\lambda}(1-y^{-2}).
\end{equation}
The polynomial $\widetilde P_{n,\lambda}$ contains only powers
congruent to $n$ modulo $2$, and
\[
[y^n]\widetilde P_{n,\lambda}(y)
=P_{n,\lambda}(1)=C_n+\lambda>0,
\]
where $P_n(1)=C_n$ follows from \eqref{eq:ThagGF}.  Thus its degree
is exactly $n$.

Let $U_m$ denote the Chebyshev polynomial of the second kind,
normalized by
\begin{equation}\label{eq:ChebyshevSecondKind}
U_m(\cos\theta)=\frac{\sin((m+1)\theta)}{\sin\theta},
\end{equation}
and put $c_j=\cos(j\pi/n)$ for $0\le j\le n$.  If
$0\le j<n/2$, then $c_j>0$, and
\eqref{eq:normalizedPRational} together with \eqref{eq:pencilSign}
gives
\[
\sgn\widetilde P_{n,\lambda}(c_j)
=\sgn P_{n,\lambda}\!\left(-\tan^2\frac{j\pi}{n}\right)
=(-1)^j.
\]
The parity of $\widetilde P_{n,\lambda}$ gives the same formula at
the negative nodes.  If $n=2d$ and $j=d$, then
\[
\widetilde P_{n,\lambda}(0)=(-1)^d a_d,
\]
whose sign is $(-1)^d$ by the positivity of the leading coefficient
proved above.  Consequently,
\begin{equation}\label{eq:normalizedNodeSign}
\sgn\widetilde P_{n,\lambda}(c_j)=(-1)^j
\qquad(0\le j\le n).
\end{equation}

Since $c_0>c_1>\cdots>c_n$, the intermediate value theorem gives at
least one zero of $\widetilde P_{n,\lambda}$ in each interval
$(c_j,c_{j-1})$, $1\le j\le n$.  Its degree is $n$, so these are
exactly its zeros and all are simple.  If they are denoted by
$\rho_1>\cdots>\rho_n$, then
\begin{equation}\label{eq:ChebyshevInterlacing}
1>\rho_1>c_1>\rho_2>c_2>\cdots>
c_{n-1}>\rho_n>-1.
\end{equation}
The zeros of $U_{n-1}$ are $c_1,\ldots,c_{n-1}$, so
\eqref{eq:ChebyshevInterlacing} says precisely that $U_{n-1}$
strictly interlaces $\widetilde P_{n,\lambda}$.

It remains to translate this back to $P_{n,\lambda}$.  If $n=2d$,
the $n$ zeros of $\widetilde P_{n,\lambda}$ occur in $d$ pairs
$\{r,-r\}$ with $0<r<1$.  If $n=2d+1$, then $0$ is one of its
zeros and the remaining zeros occur in $d$ such pairs.  By
\eqref{eq:normalizedPRational}, each pair corresponds to a single zero
\[
x=1-r^{-2}<0
\]
of $P_{n,\lambda}$.  Distinct positive values of $r$ give distinct
values of $x$.  Moreover, the change of variables
$x=1-y^{-2}$ has nonzero derivative at $y=r$, so simplicity of the
zeros of $\widetilde P_{n,\lambda}$ implies simplicity of these
zeros of $P_{n,\lambda}$.  We have therefore obtained $d$ distinct
simple negative zeros.  Since $\deg P_{n,\lambda}=d$, these are all
its zeros.
\end{proof}

\begin{proof}[Proof of Theorem~\ref{thm:KLtwoFamilies}]
Take $\lambda=0$ in Theorem~\ref{thm:KLpencil}.  For $n\ge2$, we have
$1\le n/2$, so the same theorem also applies with $\lambda=1$.  The
identity
$P_{K_{2,n}}(x)=P_n(x)+x$ of Gedeon, Proudfoot, and
Young~\cite[Theorem~5.8]{GPY} completes the proof.
\end{proof}

We record one further consequence of the normalization used in the
proof of Theorem~\ref{thm:KLpencil}.

\begin{remark}\label{rem:CatalanChebyshev}
For $\lambda=0$ and $n\ge0$, let $\widetilde P_n(y)$ be the polynomial
continuation of $y^nP_n(1-y^{-2})$.  This normalization has a direct
relation to $U_n(y/2)$.  For $n=0,1$, the two polynomials coincide;
for $n\ge2$, $\widetilde P_n(y)$ is not a scalar multiple of
$U_n(y/2)$.  The resulting sequence satisfies
\begin{align}
\sum_{n\ge0}\widetilde P_n(y)z^n
&=C(yz-z^2),
\label{eq:normalizedGF}\\
\widetilde P_n(y)
&=\sum_{m=0}^{\lfloor n/2\rfloor}
(-1)^m C_{n-m}\binom{n-m}{m}y^{n-2m}
\qquad(n\ge0).
\label{eq:CatalanChebyshevCoefficients}
\end{align}
In comparison,
\begin{equation}\label{eq:ChebyshevCoefficients}
U_n(y/2)=\sum_{m=0}^{\lfloor n/2\rfloor}
(-1)^m\binom{n-m}{m}y^{n-2m}.
\end{equation}
Thus $\widetilde P_n$ is obtained from $U_n(y/2)$ by multiplying the
coefficient of $y^{n-2m}$ by $C_{n-m}$.  It also admits the
semicircle-moment representation
\begin{equation}\label{eq:semicircleRepresentation}
\widetilde P_n(y)=\frac1{2\pi}\int_{-2}^{2}
\xi^n U_n(y\xi/2)\sqrt{4-\xi^2}\,d\xi.
\end{equation}

For $y\ne0$, replace $x$ by $1-y^{-2}$ and $z$ by $yz$ in
\eqref{eq:ThagGF}.  This gives \eqref{eq:normalizedGF}; since its
coefficients are polynomials in $y$, the identity extends to $y=0$.
Alternatively, substituting $x=1-y^{-2}$ into
\eqref{eq:PnExpansion} gives \eqref{eq:CatalanChebyshevCoefficients}
directly for $y\ne0$, and hence for every $y$ by polynomial
continuation.  The standard
explicit expansion of the Chebyshev polynomial of the second kind is
\eqref{eq:ChebyshevCoefficients}, which proves the coefficient
comparison.

Finally, the centered semicircle distribution on $[-2,2]$ satisfies
\[
\frac1{2\pi}\int_{-2}^{2}\xi^{2r}\sqrt{4-\xi^2}\,d\xi=C_r,
\qquad
\frac1{2\pi}\int_{-2}^{2}\xi^{2r+1}\sqrt{4-\xi^2}\,d\xi=0.
\]
These identities follow immediately from a beta integral.  Expanding
$U_n(y\xi/2)$ by \eqref{eq:ChebyshevCoefficients} and integrating
term by term now yields \eqref{eq:semicircleRepresentation}.
\end{remark}

We next turn from the Kazhdan--Lusztig polynomials themselves to the
associated $Z$-polynomials.  The same rational substitution will
reappear, but the coefficient estimate will now be used through a
unit-circle criterion for self-inversive polynomials.

\section{The \texorpdfstring{$Z$}{Z}-polynomials}\label{sec:Zpolynomials}

We first express $Z_{T_n}(x)$ in the $\gamma$-basis and derive the
generating function of its $\gamma$-polynomial.  We then apply the
transformation from Section~\ref{sec:KLproof}; the resulting
palindromic polynomial has a coefficient pattern to which the theorem
of Lakatos and Losonczi applies.  Finally, we transfer its zeros back
to the $\gamma$-polynomial and then to $Z_{T_n}(x)$.

\subsection{The \texorpdfstring{$\gamma$}{gamma}-polynomial}

For a matroid $M$, Proudfoot, Xu, and
Young~\cite[Definition~2.1]{ProudfootXuYoung} defined the
$Z$-polynomial by $Z_M(x)=\sum_F x^{\rk F}P_{M/F}(x)$, where $F$
ranges over the lattice $\mathcal L(M)$ of flats of $M$.  For thagomizer
matroids, the following formulas reduce the problem to the already
studied polynomials $P_r(x)$.  The first occurs in the proof of
Proposition~5.18 of Ferroni, Nasr, and
Vecchi~\cite{FerroniNasrVecchiGamma}, while the second is their
Proposition~5.20.

\begin{proposition}\label{prop:ZFormulas}
For $n\ge1$,
\begin{equation}\label{eq:ZTFormula}
Z_{T_n}(x)=x(1+x)^n+
\sum_{r=0}^n\binom nr(2x)^{n-r}P_r(x).
\end{equation}
For $n\ge2$, one has $Z_{K_{2,n}}(x)=Z_{T_n}(x)$.
\end{proposition}

For $n=0$, the graphic matroid of $T_0$ has rank $1$ and exactly two
flats.  Hence its defining sum gives $Z_{T_0}(x)=1+x$, and
\eqref{eq:ZTFormula} also holds in this case because $P_0(x)=1$.

Proudfoot, Xu, and Young~\cite[Proposition~2.3]{ProudfootXuYoung}
proved that
$Z_{T_n}(x)=x^{n+1}Z_{T_n}(x^{-1})$.  Since $Z_{T_n}(0)=1$, this
polynomial is palindromic of degree $n+1$.  For each $N\ge0$, the
polynomials
$x^k(1+x)^{N-2k}$, $0\le k\le\lfloor N/2\rfloor$, form a basis of
the vector space of polynomials palindromic with respect to $N$.
Indeed, they lie in this space and have pairwise distinct lowest
nonzero degrees, so they are linearly independent.  A palindromic
polynomial is determined by its coefficients of degrees
$0,\ldots,\lfloor N/2\rfloor$, so the space has dimension
$\lfloor N/2\rfloor+1$, equal to the number of these polynomials.
This basis therefore packages the palindromicity of $Z_{T_n}$ into a
single auxiliary polynomial.  In particular, there is a unique
polynomial $\Gamma_n(u)$ of degree at most
$\lfloor(n+1)/2\rfloor$ such that
\begin{equation}\label{eq:GammaDef}
Z_{T_n}(x)=(1+x)^{n+1}
\Gamma_n\!\left(\frac{x}{(1+x)^2}\right).
\end{equation}

Ferroni, Nasr, and Vecchi~\cite[Proposition~5.3]{FerroniNasrVecchiGamma}
showed more generally that a palindromic polynomial with positive
coefficients is real-rooted if and only if its $\gamma$-polynomial has
only negative real zeros.  Formula \eqref{eq:ZTFormula} shows that
$Z_{T_n}(x)$ has strictly positive coefficients: every matroid
Kazhdan--Lusztig polynomial has constant coefficient $1$.  In
\eqref{eq:ZTFormula}, only the summand with $r=n$ contributes to the
constant coefficient, which is therefore $P_n(0)=1$, while
$x(1+x)^n$ has positive coefficients in every degree from $1$ to
$n+1$.  Every summand in the sum has nonnegative coefficients by the
nonnegativity of matroid Kazhdan--Lusztig coefficients.  Thus
Proposition~5.3 of Ferroni, Nasr, and Vecchi applies to $Z_{T_n}$.
We shall establish the stronger fact that all zeros of $\Gamma_n$ are
simple; the final factorization will then transfer this simplicity to
$Z_{T_n}$.

The next generating function will allow the same transformation used
in Section~\ref{sec:KLproof} to be applied uniformly to all
$\Gamma_n$.

\begin{proposition}
The generating function of the polynomials $\Gamma_n$ is
\begin{equation}\label{eq:GammaGF}
\sum_{n\ge0}\Gamma_n(u)z^n
=\frac{1-\sqrt{1-4z+4(1-u)z^2}}{2z(1-z)}.
\end{equation}
\end{proposition}

\begin{proof}
Here and below, every square root occurring in a generating function
denotes the unique formal power series with constant term $1$.
Summing \eqref{eq:ZTFormula}, writing $n=r+m$, and using
\eqref{eq:ThagGF}, we obtain
\[
\sum_{n\ge0}Z_{T_n}(x)z^n
=\frac{x}{1-(1+x)z}
+\frac{1}{1-2xz}
C\!\left(\frac{z-(1+x)z^2}{(1-2xz)^2}\right).
\]
Put $u=x/(1+x)^2$.  Replace $z$ by $z/(1+x)$, divide by $1+x$, and
use \eqref{eq:GammaDef}.  This gives
\[
\sum_{n\ge0}\Gamma_n(u)z^n
=\frac{x}{(1+x)(1-z)}
+\frac{1}{1+x-2xz}
C\!\left(\frac{z(1-z)(1+x)}{(1+x-2xz)^2}\right).
\]
Writing $A=1+x-2xz$ and using
$C(w)=(1-\sqrt{1-4w})/(2w)$, note that the argument of $C$ is
$w=z(1-z)(1+x)/A^2$.  Since $A/(1+x)$ has constant term $1$ as a
power series in $z$, the normalized formal square roots satisfy
\[
\sqrt{1-4w}
=\frac{1+x}{A}
\sqrt{\frac{A^2-4z(1-z)(1+x)}{(1+x)^2}}.
\]
It follows that the right-hand side becomes
\[
\frac{1-
\sqrt{\bigl(A^2-4z(1-z)(1+x)\bigr)/(1+x)^2}}
{2z(1-z)}.
\]
Since
\[
\frac{(1+x-2xz)^2-4z(1-z)(1+x)}{(1+x)^2}
=1-4z+4(1-u)z^2,
\]
formula \eqref{eq:GammaGF} follows.  The substitution
$u=x/(1+x)^2$ induces an injective homomorphism
$\mathbb R[u]\to\mathbb R(x)$, so the coefficient identity holds in
$\mathbb R[u][[z]]$.
\end{proof}

\subsection{The transformed \texorpdfstring{$\gamma$}{gamma}-polynomials}

Apply the transformation in Lemma~\ref{lem:PhiBasic} to
$\Gamma_n$ and set $D_n(q)=\Phi_{n+1}(\Gamma_n)(q)$ for $n\ge0$.
For $q\ne-1$, formulas \eqref{eq:PhiRational} and
\eqref{eq:GammaGF}, together with
$1-((q-1)/(q+1))^2=4q/(1+q)^2$, give
\begin{equation}\label{eq:DnGF}
\sum_{n\ge0}D_n(q)z^n
=\frac{1-\sqrt{(1-4z)(1-4qz)}}
{2z(1-(1+q)z)}.
\end{equation}
The coefficient of every power of $z$ on both sides is a polynomial
in $q$, so \eqref{eq:DnGF} holds for every $q$.  To apply a
unit-circle theorem, we must now identify the signs of the interior
coefficients of $D_n$.

\begin{lemma}\label{lem:DnCoeffSign}
For every $n\ge0$,
\begin{equation}\label{eq:DnCoeffShape}
D_n(q)=\rho_n(1+q^{n+1})-
\sum_{r=1}^n\delta_{n,r}q^r,
\end{equation}
where $\delta_{n,r}\ge0$ and
$\rho_n=\sum_{j=0}^nC_j$.
\end{lemma}

\begin{proof}
For $r\ge0$, set
$\mathcal D_r(z)=[q^r]\sum_{n\ge0}D_n(q)z^n$.  Setting $q=0$ in
\eqref{eq:DnGF} gives
$\mathcal D_0(z)=C(z)/(1-z)$.  Hence the constant coefficient of
$D_n(q)$ is $\rho_n$.  Since
$\sqrt{1-4qz}=1-2\sum_{r\ge1}C_{r-1}(qz)^r$, multiplying
\eqref{eq:DnGF} by $1-(1+q)z$ and extracting the coefficient of
$q^r$ gives
\begin{equation}\label{eq:DrRec}
\mathcal D_r(z)=\frac{z\mathcal D_{r-1}(z)+C_{r-1}z^{r-1}\sqrt{1-4z}}{1-z}
\qquad (r\ge1).
\end{equation}
We claim that
\begin{equation}\label{eq:DrSignForm}
\mathcal D_r(z)=\rho_{r-1}z^{r-1}-z^rE_r(z),
\qquad E_r(z)\in\mathbb R_{\ge0}[[z]].
\end{equation}
For $r=1$, the identities
$\sqrt{1-4z}=1-2zC(z)$ and $C(z)-1=zC(z)^2$ give
\[
\begin{aligned}
\mathcal D_1(z)
&=\frac{zC(z)+(1-z)\sqrt{1-4z}}{(1-z)^2}\\
&=\frac{(1-z)^2-z^2(C(z)-1)^2}{(1-z)^2}\\
&=1-\frac{z^4C(z)^4}{(1-z)^2}
 =1-zE_1(z),\\
E_1(z)
&=\frac{z^3C(z)^4}{(1-z)^2}
 \in\mathbb R_{\ge0}[[z]].
\end{aligned}
\]
Thus \eqref{eq:DrSignForm} holds for $r=1$.  Suppose that it holds for
$r-1$, where $r\ge2$.  Using
$\rho_{r-2}+C_{r-1}=\rho_{r-1}$, substitution in
\eqref{eq:DrRec} gives
\[
\mathcal D_r(z)=\frac{\rho_{r-1}z^{r-1}
-z^r(E_{r-1}(z)+2C_{r-1}C(z))}{1-z}.
\]
Therefore \eqref{eq:DrSignForm} holds with
\[
E_r(z)=\frac{E_{r-1}(z)+2C_{r-1}C(z)-\rho_{r-1}}{1-z}.
\]
It remains to prove that $E_r(z)$ has nonnegative coefficients.  We have
\[
E_{r-1}(z)+2C_{r-1}C(z)-\rho_{r-1}
=E_{r-1}(z)+2C_{r-1}(C(z)-1)
 +(2C_{r-1}-\rho_{r-1}).
\]
The first two terms on the right have nonnegative coefficients.
Moreover,
$\rho_0\le2C_0$, $\rho_1=2C_1$, and
$C_k=((4k-2)/(k+1))C_{k-1}\ge2C_{k-1}$ for $k\ge2$.  Hence, if
$\rho_{k-1}\le2C_{k-1}$, then
$\rho_k=\rho_{k-1}+C_k\le2C_{k-1}+C_k\le2C_k$.  Thus
$\rho_k\le2C_k$ for every $k\ge0$.  Therefore the numerator in the
definition of $E_r(z)$ has nonnegative
coefficients.  Since $(1-z)^{-1}$ also has nonnegative coefficients,
$E_r(z)$ is coefficientwise nonnegative.

Taking $r=n+1$ in \eqref{eq:DrSignForm} gives
$[q^{n+1}]D_n(q)=[z^n]\mathcal D_{n+1}(z)=\rho_n$.  For
$1\le r\le n$, the same formula gives
\[
[q^r]D_n(q)=[z^n]\mathcal D_r(z)
=-[z^{n-r}]E_r(z)\le0.
\]
Together with the constant coefficient $\rho_n$, this proves
\eqref{eq:DnCoeffShape}.
\end{proof}

Lemma~\ref{lem:DnCoeffSign} shows that $D_n$ has positive constant and
leading coefficients and nonpositive interior coefficients.  The
value at
$q=1$ will make the resulting coefficient inequality strict, placing
$D_n$ within the following unit-circle criterion.

\subsection{Zeros of the \texorpdfstring{$\gamma$}{gamma}- and \texorpdfstring{$Z$}{Z}-polynomials}

A polynomial $f(q)=\sum_{j=0}^m a_jq^j$ is called
\emph{self-inversive} if
$f(q)=\omega q^m\overline{f(1/\overline q)}$ for some
$|\omega|=1$.  Equivalently,
$a_j=\omega\overline{a_{m-j}}$.  Every real palindromic polynomial is
self-inversive with $\omega=1$.

We shall use the following theorem of Lakatos and
Losonczi~\cite[Theorem~1(ii)-1 and (ii)-2]{LakatosLosonczi}.

\begin{theorem}[Lakatos--Losonczi]\label{thm:LL}
Let $m\ge1$, and let
$f(q)=\sum_{j=0}^m a_jq^j$ be self-inversive of degree $m$.  If
\[
|a_m|\ge\frac12\sum_{j=1}^{m-1}|a_j|,
\]
then all zeros of $f$ lie on the unit circle.  If the inequality is
strict, then all zeros are simple.
\end{theorem}

\begin{theorem}\label{thm:GammaRealRooted}
For every $n\ge1$, the polynomial $\Gamma_n(u)$ has degree
$\lfloor(n+1)/2\rfloor$ and has only simple negative zeros.
\end{theorem}

\begin{proof}
Setting $u=0$ in \eqref{eq:GammaGF} gives
$\sum_{n\ge0}\Gamma_n(0)z^n=1/(1-z)$, so $\Gamma_n(0)=1$ and hence
$D_n(1)=2^{n+1}\Gamma_n(0)=2^{n+1}$.  By
Lemma~\ref{lem:DnCoeffSign}, $D_n$ has degree $n+1$, with constant and
leading coefficients both equal to $\rho_n$.
Evaluating \eqref{eq:DnCoeffShape} at $q=1$ gives
$2\rho_n-\sum_{r=1}^n\delta_{n,r}=2^{n+1}>0$, and hence
$\rho_n>\frac12\sum_{r=1}^n\delta_{n,r}$.  The coefficients of
$\Gamma_n$ are real, and hence so are those of
$D_n$.  Lemma~\ref{lem:PhiBasic} shows that $D_n$ is palindromic.
In Theorem~\ref{thm:LL}, take $m=n+1$,
$a_0=a_{n+1}=\rho_n$, and $a_r=-\delta_{n,r}$ for
$1\le r\le n$.  The strict coefficient inequality above is exactly
the hypothesis of that theorem.  Therefore all zeros of $D_n$ are
simple and lie on the unit circle.  We now determine the degree of
$\Gamma_n$ and transfer these zeros through the defining rational
map.

Put $\phi(q)=((q-1)/(q+1))^2$, and let
$e=\deg\Gamma_n$.  Write
$\Gamma_n(u)=\sum_{k=0}^e\gamma_{n,k}u^k$.  Then
\[
D_n(q)=(1+q)^{n+1-2e}H_n(q),
\qquad
H_n(q)=\sum_{k=0}^e\gamma_{n,k}(q-1)^{2k}(q+1)^{2e-2k}.
\]
For $q\ne-1$, we have
$H_n(q)=(q+1)^{2e}\Gamma_n(\phi(q))$, while
$[q^{2e}]H_n(q)=\sum_{k=0}^e\gamma_{n,k}=\Gamma_n(1)$.  Since
$D_n(q)=(1+q)^{n+1-2e}H_n(q)$ and $\deg H_n\le2e$, the term
$q^{n+1}$ can arise only by multiplying the leading term
$q^{n+1-2e}$ of $(1+q)^{n+1-2e}$ by the $q^{2e}$-term of $H_n$.
Therefore
$[q^{n+1}]D_n(q)=[q^{2e}]H_n(q)=\Gamma_n(1)$.  On the other hand,
$[q^{n+1}]D_n(q)=\rho_n$, so $\Gamma_n(1)=\rho_n>0$ and consequently
$\deg H_n=2e$.
Furthermore,
$H_n(-1)=2^{2e}\gamma_{n,e}\ne0$.  Hence $q=-1$ has multiplicity
$n+1-2e$ in $D_n$.  Since all zeros of $D_n$ are simple,
$n+1-2e\le1$.  This nonnegative integer has the same parity as
$n+1$.  Together with $0\le e\le\lfloor(n+1)/2\rfloor$, this forces
$e=\lfloor(n+1)/2\rfloor$.

Neither $0$ nor $1$ is a zero of $\Gamma_n$, since
$\Gamma_n(0)=1$ and $\Gamma_n(1)=\rho_n$.  If $u$ is a zero of
$\Gamma_n$, then $\phi(q)=u$ is equivalent to
$(1-u)q^2-2(1+u)q+(1-u)=0$.
Here the leading coefficient $1-u$ is nonzero and the discriminant is
$16u\ne0$, so the equation has two distinct finite roots, whose product
is $1$.  Moreover, neither root is $\pm1$, and
$\phi'(q)=4(q-1)/(q+1)^3\ne0$ at either root.  Thus a zero of
$\Gamma_n$ of multiplicity $m$ gives two zeros of $H_n$, each of
multiplicity $m$.  Distinct zeros of $\Gamma_n$ have disjoint sets of
preimages under $\phi$.  Since the multiplicities of the zeros of
$\Gamma_n$ sum to $e$ and $\deg H_n=2e$, these account for all zeros
of $H_n$.

Since
$D_n(q)=(1+q)^{n+1-2e}H_n(q)$, every zero of $H_n$ is a zero of
$D_n$.  Moreover, $H_n(-1)\ne0$, so the zeros of $H_n$ are simple
and lie on the unit circle.  Let $u$ be any zero of $\Gamma_n$, and
let $q$ be one of its two preimages under $\phi$.  Then $q$ is a zero
of $H_n$ and $q\ne\pm1$.  Choose
$\theta\in(-\pi/2,\pi/2)$ such that $q=e^{2i\theta}$.  Since
$\theta\ne0$, we have $u=\phi(q)=-\tan^2\theta<0$.  Therefore all
zeros of $\Gamma_n$ are negative and real.  They are simple because a
multiple zero of $\Gamma_n$ would give a multiple zero of $H_n$.
\end{proof}

It remains to translate the zero structure of $\Gamma_n$ back through
the $\gamma$-expansion \eqref{eq:GammaDef}.  This last step reduces to
the roots of explicit reciprocal quadratic factors.

\begin{proof}[Proof of Theorem~\ref{thm:ZRealRooted}]
Let $n\ge2$.  By Proposition~\ref{prop:ZFormulas}, it suffices to
consider $Z_{T_n}(x)$.  By Theorem~\ref{thm:GammaRealRooted}, if
$e=\deg\Gamma_n$, then
$\Gamma_n(u)=c\prod_{j=1}^e(u+a_j)$, where
$a_1,\ldots,a_e$ are positive and distinct.  Since
$1=\Gamma_n(0)=c\prod_{j=1}^e a_j$, we have $c>0$.  Formula
\eqref{eq:GammaDef} gives
\[
Z_{T_n}(x)=c(1+x)^{n+1-2e}
\prod_{j=1}^e\bigl(a_jx^2+(2a_j+1)x+a_j\bigr).
\]
The discriminant of the $j$th quadratic factor is $4a_j+1>0$.
The product of its roots is $1$, and their sum is negative; hence the
two roots are distinct and negative.  Quadratic factors corresponding
to distinct values of $a_j$ have no common root, since
$a=-x/(1+x)^2$ is determined by $x\ne-1$.
Also, $x=-1$ is not a zero of any quadratic factor.  Since
$e=\lfloor(n+1)/2\rfloor$, the factor $1+x$ occurs with exponent $1$
when $n$ is even and with exponent $0$ when $n$ is odd.  Thus, in the
even case it contributes the additional simple root $-1$, while in
the odd case there is no such factor.  In either case, all $n+1$
zeros of $Z_{T_n}(x)$ are distinct and negative.
\end{proof}

\section*{Acknowledgements}
The author thanks Matthew H. Y. Xie and Arthur L. B. Yang for helpful
discussions.  This work was supported by the National Natural Science
Foundation of China (No.~12171362) and the Tianjin Municipal Natural
Science Foundation (No.~25JCYBJC00430).


\begin{thebibliography}{99}

\bibitem{BradenHuhMatherneProudfootWang}
T. Braden, J. Huh, J. P. Matherne, N. Proudfoot, and B. Wang,
\newblock Singular Hodge theory for combinatorial geometries,
\newblock \emph{J. Amer. Math. Soc.}, to appear.

\bibitem{BradenProudfootICM}
T. Braden and N. Proudfoot,
\newblock Intersection cohomology without spaces,
\newblock to appear in the \emph{Proceedings of the International Congress of Mathematicians 2026}.

\bibitem{ChengLiuNonunimodal}
R. Cheng and S. Liu,
\newblock Kazhdan--Lusztig polynomials of matroids need not be unimodal,
\newblock arXiv:2607.24186v2, 2026.

\bibitem{EliasProudfootWakefield}
B. Elias, N. Proudfoot, and M. Wakefield,
\newblock The Kazhdan--Lusztig polynomial of a matroid,
\newblock \emph{Adv. Math.} \textbf{299} (2016), 36--70.

\bibitem{FerroniNasrVecchiGamma}
L. Ferroni, G. D. Nasr, and L. Vecchi,
\newblock Stressed hyperplanes and Kazhdan--Lusztig $\gamma$-positivity for matroids,
\newblock \emph{Int. Math. Res. Not. IMRN} (2023), no.~24, 20883--20942.

\bibitem{GaoLuXieYangZhangUniform}
A. L. L. Gao, L. Lu, M. H. Y. Xie, A. L. B. Yang, and P. B. Zhang,
\newblock The Kazhdan--Lusztig polynomials of uniform matroids,
\newblock \emph{Adv. in Appl. Math.} \textbf{122} (2021), Paper No.~102117.

\bibitem{Gedeon}
K. R. Gedeon,
\newblock Kazhdan--Lusztig polynomials of thagomizer matroids,
\newblock \emph{Electron. J. Combin.} \textbf{24} (2017), no.~3, Paper No.~P3.12.

\bibitem{GPY}
K. Gedeon, N. Proudfoot, and B. Young,
\newblock Kazhdan--Lusztig polynomials of matroids: a survey of results and conjectures,
\newblock \emph{S\'em. Lothar. Combin.} \textbf{78B} (2017), Art.~80.

\bibitem{GesselSuperBallot}
I. M. Gessel,
\newblock Super ballot numbers,
\newblock \emph{J. Symbolic Comput.} \textbf{14} (1992), no.~2--3, 179--194.

\bibitem{LakatosLosonczi}
P. Lakatos and L. Losonczi,
\newblock Self-inversive polynomials whose zeros are on the unit circle,
\newblock \emph{Publ. Math. Debrecen} \textbf{65} (2004), no.~3--4, 409--420.

\bibitem{ProudfootXuYoung}
N. Proudfoot, Y. Xu, and B. Young,
\newblock The $Z$-polynomial of a matroid,
\newblock \emph{Electron. J. Combin.} \textbf{25} (2018), no.~1, Paper No.~P1.26.

\bibitem{Vella}
A. Vella,
\newblock Pattern avoidance in permutations: linear and cyclic orders,
\newblock \emph{Electron. J. Combin.} \textbf{9} (2002/03), no.~2,
Research Paper~R18.

\bibitem{WuZhangThagomizerLogConcavity}
S. Wu and P. B. Zhang,
\newblock The log-concavity of Kazhdan--Lusztig polynomials of thagomizer matroids,
\newblock \emph{Discrete Math.} \textbf{346} (2023), no.~7, Paper No.~113381.

\end{thebibliography}
\end{document}